\documentclass[12pt]{amsart}

\usepackage[OT1]{fontenc}
\usepackage{geometry}
\usepackage[psamsfonts]{amssymb}
\usepackage{amsfonts,amsmath,amstext,amsbsy,amsopn,amsthm}
\usepackage{mathtools}
\usepackage{upref,xcolor}
\usepackage{url}
\usepackage{helvet,courier,type1cm}
\usepackage[bottom]{footmisc}
\usepackage[colorlinks,citecolor=red,linkcolor=blue,hyperindex]{hyperref}
\usepackage{bookmark}

\allowdisplaybreaks[2]

\hypersetup{
  pdftitle={Besov spaces and product BMO in arbitrary dimension},
  pdfauthor={Chen Li and Kai Wang},
  pdfsubject={Besov embedding and Fejer approximation of bounded Hankel operators}
}

\newtheorem{theorem}{Theorem}[section]
\newtheorem{lemma}[theorem]{Lemma}
\newtheorem{proposition}[theorem]{Proposition}
\newtheorem{corollary}[theorem]{Corollary}

\theoremstyle{definition}
\newtheorem{question}[theorem]{Question}

\theoremstyle{remark}
\newtheorem{remark}[theorem]{Remark}

\newcommand{\R}{\mathbb{R}}
\newcommand{\C}{\mathbb{C}}
\newcommand{\Z}{\mathbb{Z}}
\newcommand{\D}{\mathcal{D}}
\newcommand{\Qp}{\mathbb{R}_+^n}

\newcommand{\Bp}{\mathcal{B}_p}
\newcommand{\BMO}{\mathrm{BMO}_{\mathrm{prod}}}
\newcommand{\Hprod}{H^1_{\mathrm{prod}}}

\newcommand{\supp}{\mathrm{supp}}
\DeclareMathOperator*{\essinf}{ess\,inf}

\newcommand{\norm}[1]{\left\lVert#1\right\rVert}
\newcommand{\ip}[2]{\langle#1,#2\rangle}
\newcommand{\jj}{\mathbf{j}}
\newcommand{\kk}{\mathbf{k}}
\newcommand{\mm}{\mathbf{m}}
\newcommand{\ind}{\mathbf{1}}

\numberwithin{equation}{section}

\begin{document}

\title[Besov spaces and product BMO]
{Besov spaces and product BMO over polydisc}

\author{Chen Li}
\address{Chen Li: School of Mathematical Sciences, Fudan University,
Shanghai, China}
\email{22110180024@m.fudan.edu.cn}

\author{Kai Wang}
\address{Kai Wang: School of Mathematical Sciences, Fudan University,
Shanghai, China}
\email{kwang@fudan.edu.cn}

\date{October 8, 2026}

\begin{abstract}
We prove that the analytic Besov space
$B_{p,p}^{1/p}(\C_+^n)$ embeds continuously into Chang--Fefferman product
BMO for every finite $n$ and $1\le p<\infty$. Consequently, every
small Hankel operator in a finite Schatten class has a product BMO
symbol. The proof uses smooth product wavelets and an estimate for
dyadic rectangles. We also give
a condition on Fej\'er means that yields a BMO symbol for a bounded
Hankel operator on the polydisc.
\end{abstract}

\maketitle

\section{Introduction}

For a Hankel operator, one would like to describe boundedness and
compactness through its symbol. On the Hardy space $H^2(\mathbb D)$
of the unit disc, Nehari's theorem \cite{Nehari} says that boundedness
is equivalent to the existence of a $\mathrm{BMOA}$ symbol. For compact
operators, Hartman's theorem \cite{Hartman} replaces 
by $\mathrm{VMOA}$ ones.
 Peller \cite{Peller1,Peller2} proved that, for
$0<p<\infty$, a Hankel operator belongs to $S^p$ if and only if
its analytic symbol belongs to the Besov space $B_{p,p}^{1/p}$.

The corresponding questions for small Hankel operators in several
variables are more difficult. One expects Chang--Fefferman product
BMO and its vanishing subspace VMO to describe boundedness and
compactness. A BMO characterization was claimed on the bidisc by
Ferguson and Lacey \cite{FL}, and in several variables by Lacey and
Terwilleger \cite{LT}. A VMO characterization of compactness on the
bidisc was also stated by Lacey, Terwilleger, and Wick \cite{LTW}.
However, Holmes, Treil, and Volberg \cite[Section 10]{HTV} showed
that a lower-bound estimate used in the boundedness proof fails.
The compactness proof in \cite{LTW} uses that estimate as well.
This leaves a gap in these arguments, rather than a counterexample
to the proposed characterizations. The general Nehari problem
remains open; see \cite{HTV,BPBesov,BPRecover}.

The finite Schatten classes offer a way to study this problem under
a stronger assumption on the operator. Here a Besov description is
already available. Pott and Smith \cite{PS} used the vector-valued
form of Peller's theorem to characterize small Hankel operators in
$S^p$ in several variables when $1<p<\infty$. Bampouras and Perfekt
\cite{BPBesov} proved the characterization for $1\le p<\infty$ on
product Hardy spaces and also treated weighted operators. These
results reduce the question of a BMO symbol for an $S^p$ Hankel
operator to an embedding of its Besov symbol space into product BMO.

Bampouras and Perfekt \cite{BPRecover} proved this embedding in two
variables. Thus the conclusion of Nehari's theorem holds in that
case whenever the operator belongs to some $S^p$, $1\le p<\infty$.
They noted that three or more variables require further proof
because of the geometry of product rectangles. We prove the
embedding in every finite dimension. The argument uses smooth
product wavelets and a sum of powers of the measures of dyadic
rectangles.

Let $\C_+=\{z\in\C:\operatorname{Im}z>0\}$ and
$\Qp=(0,\infty)^n$. We identify $H^2(\C_+^n)$ with its boundary
functions in $L^2(\R^n)$, whose Fourier transforms are supported in
$\overline{\Qp}$. Let $P_+$ be the Fourier projection onto this
positive orthant, and write $\BMO$ for Chang--Fefferman product BMO
on $\R^n$. Its definition is recalled in Section \ref{subsec:wavelets}. A small Hankel operator with symbol
$\varphi$ is given, initially on a suitable test class, by
\begin{equation}\label{eq:hankel}
 \widehat{H_\varphi u}(\xi)
 =\int_{\Qp}\widehat\varphi(\xi+\lambda)\widehat u(\lambda)\,d\lambda,
 \qquad \xi\in\Qp.
\end{equation}
\begin{samepage}
Write $\Bp=B_{p,p}^{1/p}(\C_+^n)$ for the analytic Besov
space. Our main result is the following embedding.

\begin{theorem}\label{thm:embedding}
Let $n\ge1$ and $1\le p<\infty$. There is a continuous linear embedding
\[
 \Bp\longrightarrow P_+\BMO
\]
with
\begin{equation}\label{eq:mainembedding}
 \norm f_{\BMO}\le_{n,p}\norm f_{\Bp}.
\end{equation}
\end{theorem}
\end{samepage}

The Schatten--Besov characterization in
\cite[Theorem 6.2(2)]{BPBesov} gives
$\norm\varphi_{\Bp}\asymp_{n,p}\norm{H_\varphi}_{S^p}$.
Theorem \ref{thm:embedding} therefore has the following consequence.

\begin{corollary}\label{cor:hankel}
Let $n\ge1$ and $1\le p<\infty$. If
$H_\varphi\in S^p(H^2(\C_+^n))$, then its analytic symbol $\varphi\in\BMO$, and
\begin{equation}\label{eq:mainhankel}
 \norm\varphi_{\BMO}\le_{n,p}\norm{H_\varphi}_{S^p}.
\end{equation}
\end{corollary}

We use the following estimate for dyadic rectangles.

\begin{proposition}\label{prop:intropacking}
Let $n\ge1$ and $s>1$, and let $\D^n$ denote the family of dyadic
rectangles in $\R^n$. For every open set $\Omega\subset\R^n$ of
finite measure,
\begin{equation}\label{eq:intropacking}
 \sum_{\substack{R\in\D^n\\R\subset\Omega}}|R|^s
 \lesssim_{n,s}|\Omega|^s.
\end{equation}
\end{proposition}

We first estimate the $\ell^p$ norm of the normalized wavelet
coefficients by the Besov norm. When $p>2$, H\"older's inequality
reduces the BMO estimate to Proposition \ref{prop:intropacking}
with $s=p/(p-2)$. When $1\le p\le2$, the inclusion
$\ell^p\subset\ell^2$ gives the estimate directly.
Section \ref{subsec:obstruction} gives a counterexample to the
direct higher-dimensional form of the one-coordinate covering
estimate used in \cite{BPRecover}.Section 3 and 4 is devoted to the proof of main theorem. Finally,
Section \ref{sec:fejer} considers Fej\'er means of symbols on the
polydisc. It explains which further estimates would allow the
BMOA and VMOA conclusions for polynomial symbols to pass to
bounded and compact Hankel operators.

\section{Preliminaries}

\subsection{Besov spaces and Schatten Hankel operators}
\label{subsec:besov}

We write $|E|$ for the Lebesgue measure of a set $E$ and
$\ind_E$ for its indicator function. The notation $\norm f_p$ means
the $L^p$ norm, and $\norm{(b_j)}_{\ell^p}=(\sum_j|b_j|^p)^{1/p}$.
We use $A\lesssim_{n,p}B$ when $A\le C_{n,p}B$ for a constant
depending only on the indicated parameters; $A\asymp_{n,p}B$ means
that both inequalities hold. The fixed cutoffs and wavelets may
also enter these constants.

For a compact operator $T$, let $s_j(T)$ be its singular values.
The Schatten class $S^p$ consists of those $T$ for which
$\norm T_{S^p}=(\sum_j s_j(T)^p)^{1/p}<\infty$.
We denote the operator norm by $\norm T_{\mathrm{op}}$.

We use the Fourier transform
$\widehat f(\xi)=\int_{\R^n}f(x)e^{-2\pi i x\cdot\xi}\,dx$
and the pairing $\ip f g=\int f\overline g$, linear in the first
variable. For $\jj=(j_1,\ldots,j_n)\in\Z^n$, put
$\sigma(\jj)=j_1+\cdots+j_n$.
Choose a real smooth function $\psi$, compactly supported in
$(0,\infty)$, with
\[
 \supp\psi\subset[1/2,2],
 \qquad \sum_{j\in\Z}\psi(2^{-j}\xi)=1\quad(\xi>0).
\]
For $\xi=(\xi_1,\ldots,\xi_n)$, set
\[
 \Psi_\jj(\xi)=\prod_{\ell=1}^n\psi(2^{-j_\ell}\xi_\ell),
 \qquad
 \Delta_\jj f=\mathcal F^{-1}(\Psi_\jj\widehat f).
\]
The Besov norm is defined as
\begin{equation}\label{eq:besov}
 \norm f_{\Bp}^p
 =\sum_{\jj\in\Z^n}2^{\sigma(\jj)}
      \norm{\Delta_\jj f}_p^p.
\end{equation}
Here $\mathcal F^{-1}$ is the inverse Fourier transform, and
$j_\ell$ records the frequency scale in the $\ell$th coordinate.
The space $\Bp$ consists of boundary functions of analytic functions
on $\C_+^n$ for which \eqref{eq:besov} is finite. Each frequency point belongs to only finitely many
blocks, and the convolution kernels of $\Delta_\jj$ have uniformly
bounded $L^1$ norms.

In this notation, the characterization of Bampouras and Perfekt
\cite[Theorem 6.2(2)]{BPBesov} states that, for $1\le p<\infty$,
\begin{equation}\label{eq:schattenbesov}
 H_\varphi\in S^p(H^2(\C_+^n))
 \quad\Longleftrightarrow\quad
 \varphi\in\Bp,
 \qquad
 \norm\varphi_{\Bp}\asymp_{n,p}\norm{H_\varphi}_{S^p}.
\end{equation}

\subsection{Smooth wavelets and product BMO}\label{subsec:wavelets}

Let
\[
 \D=\{2^{-k}[m,m+1):k,m\in\Z\},
 \qquad \D^n=\{I_1\times\cdots\times I_n:I_\ell\in\D\}.
\]
Fix a real Meyer wavelet $w$, normalized in $L^2(\R)$, with
\[
 w\in\mathcal S(\R),\qquad
 \supp\widehat w\subset\{\xi:a\le|\xi|\le b\},
 \qquad 0<a<b<\infty.
\]
Here $\mathcal S(\R)$ is the Schwartz class, which resembles the function class that all
its derivatives decay faster than every negative power of $|x|$.
Write $W(x)=\prod_{\ell=1}^nw(x_\ell)$ and, for
$\kk,\mm\in\Z^n$, set
\[
 w_{\kk,\mm}(x)
 =2^{\sigma(\kk)/2}W(2^\kk x-\mm),
 \qquad
 R_{\kk,\mm}
 =\prod_{\ell=1}^n2^{-k_\ell}[m_\ell,m_\ell+1),
\]
with
$2^\kk x=(2^{k_1}x_1,\ldots,2^{k_n}x_n)$.
These functions form an orthonormal basis of $L^2(\R^n)$.
We write $w_R=w_{\kk,\mm}$ when $R=R_{\kk,\mm}$ for abbreviation. Thus,
$|R|=2^{-\sigma(\kk)}$.

Chang--Fefferman product BMO is the dual of the real product Hardy
space $\Hprod$. Thus a BMO function $f$ acts by
$g\mapsto\ip f g$, bounded with respect to the $\Hprod$ norm.
We use the usual identification of functions in these homogeneous
spaces. The smooth-wavelet characterization of BMO is
\begin{equation}\label{eq:carleson}
 \norm f_{\BMO}\asymp_n
 \sup_{\substack{\Omega\subset\R^n\ {\rm open}\\0<|\Omega|<\infty}}
 \left(\frac1{|\Omega|}
       \sum_{\substack{R\in\D^n\\R\subset\Omega}}
          |\ip f{w_R}|^2\right)^{1/2}.
\end{equation}
Conversely, a sequence $(c_R)$ satisfying the bound on the right
is the wavelet coefficient sequence of a BMO function, with its norm
controlled by this bound.
See
\cite[Sections 2.1--2.2, formula (2.4) and Theorem 2.5]{DP},
with $n$ one-dimensional factors.

The analytic projection
\[
 P_+=\prod_{\ell=1}^n\frac{I+i\mathcal H_\ell}{2}
\]
where $\mathcal H_\ell$ is the partial Hilbert transform, is bounded
on $\Hprod$ and on $\BMO$. On $L^2$ it is the self-adjoint Fourier
projection to $\Qp$.

\subsection{The fail of covering estimate}
\label{subsec:obstruction}

Let $\Omega\subset\R^n$ be open and have finite measure. A dyadic
rectangle $R=I_1\times\cdots\times I_n\subset\Omega$ is maximal in
the $r$th coordinate if
\[
 I_r\subset J,\qquad
 I_1\times\cdots\times I_{r-1}\times J\times I_{r+1}
       \times\cdots\times I_n\subset\Omega,\qquad J\in\D,
\]
imply $J=I_r$. Thus the other sides are kept fixed when testing
maximality. We denote this family by $\mathcal M_r(\Omega)$.

Put
\[
 \widetilde\Omega=\{M_s\ind_\Omega>1/2\},
\]
where the uncentered strong maximal operator is defined by
\[
 M_sh(x)=\sup_{\substack{Q\ni x\\Q\text{ axis-parallel rectangle}}}
             \frac1{|Q|}\int_Q|h(y)|\,dy.
\]
The rectangles in this supremum need not be dyadic.
For $R=I_1\times\cdots\times I_n\subset\Omega$, let $\widehat I_1$
be the largest dyadic ancestor of $I_1$ such that
$\widehat I_1\times I_2\times\cdots\times I_n\subset\widetilde\Omega$,
and set $\gamma_1(R)=|\widehat I_1|/|I_1|$.
Here $\Omega\subset\widetilde\Omega$ and
$|\widetilde\Omega|\lesssim_n|\Omega|$, so such a largest ancestor
exists. The number $\gamma_1(R)$ measures how far the first side
can be enlarged inside $\widetilde\Omega$.

In two dimensions, the form of Journ\'e's lemma used in
\cite[Lemma 2.2]{BPRecover} is
\begin{equation}\label{eq:journe2}
 \sum_{R\in\mathcal M_2(\Omega)}
       \gamma_1(R)^{-\delta}|R|
 \lesssim_\delta|\Omega|,\qquad \delta>0.
\end{equation}
Replacing $\mathcal M_2$ by $\mathcal M_n$ gives a false estimate
when $n\ge3$.

\begin{proposition}\label{prop:failure}
Let $n\ge3$ and $\delta>0$. There is no constant $C_{n,\delta}$ such that
\begin{equation}\label{eq:false}
 \sum_{R\in\mathcal M_n(\Omega)}
       \gamma_1(R)^{-\delta}|R|
 \le C_{n,\delta}|\Omega|
\end{equation}
for every finite-measure open set $\Omega\subset\R^n$.
\end{proposition}

\begin{proof}
Take $A=(-1/10,11/10)$, $\Omega=A^n$, and $I=[0,1)$.
For every dyadic subinterval $J\subset I$, consider
\[
 R_J=I\times J\times I^{\,n-2}.
\]
The last side is maximal in $\Omega$, since the dyadic parent
$[0,2)$ of $I$ is not contained in $A$. Hence
$R_J\in\mathcal M_n(\Omega)$.

The numbers $\gamma_1(R_J)$ are bounded above independently of $J$.
To see this, let $M$ be the one-dimensional uncentered maximal
operator. For every $x$,
\[
 M_s\ind_{A^n}(x)\le M\ind_A(x_1).
\]
The set $\{M\ind_A>1/2\}$ is a bounded interval. It follows that the
first side of any enlarged rectangle defining $\gamma_1(R_J)$ has
uniformly bounded length. Since $|I|=1$, we have
$1\le\gamma_1(R_J)\le C$.

There are $2^k$ choices of $J$ of length $2^{-k}$, and $|R_J|=|J|$.
Consequently, for every $N$,
\[
 \sum_{k=0}^N
 \sum_{\substack{J\subset I\\|J|=2^{-k}}}
       \gamma_1(R_J)^{-\delta}|R_J|
 \ge C^{-\delta}\sum_{k=0}^N2^k2^{-k}
 =C^{-\delta}(N+1).
\]
This contradicts \eqref{eq:false}, since $|\Omega|$ is fixed.
\end{proof}

The second side can be subdivided at every scale without changing
maximality in the last coordinate or increasing $\gamma_1$. The
factor $\gamma_1^{-\delta}$ therefore does not control this sum.
This is the obstruction to using \eqref{eq:journe2} in the same form
when $n\ge3$.

\section{An estimate for dyadic rectangles}
\label{sec:packing}

For a measurable function $g$ and a set $E$ of positive measure,
$\essinf_Eg$ denotes the largest number $c$ such that $g\ge c$
almost everywhere on $E$.

\begin{lemma}\label{lem:one}
Let $s>1$ and $g\ge0$ belong to $L^1(\R)$. Then
\begin{equation}\label{eq:one}
 \sum_{I\in\D}|I|^s(\essinf_I g)^s
 \le A_s\norm g_1^s,
 \qquad A_s=\frac{s}{1-2^{1-s}}.
\end{equation}
\end{lemma}

\begin{proof}
Put $M=\norm g_1$. If $M=0$, then $g=0$ almost everywhere and the
claim follows. We therefore assume $M>0$.
For a fixed dyadic interval $J$, its $k$th generation consists of
$2^k$ intervals, each of length $2^{-k}|J|$. Hence
\begin{equation}\label{eq:descendants}
 \sum_{\substack{I\in\D\\I\subset J}}|I|^s
 =|J|^s\sum_{k=0}^\infty 2^k2^{-ks}
 =\frac{|J|^s}{1-2^{1-s}}.
\end{equation}

Fix $t>0$ and let $E_t=\{g>t\}$. Since
$t|E_t|\le M$, this set has finite measure. Consider the dyadic
intervals $I$ for which $|I\setminus E_t|=0$. Every such interval
has length at most $|E_t|$. Its dyadic ancestors have lengths tending
to infinity, so there is a largest ancestor that is still contained
in $E_t$ up to a null set. Let $\mathcal J_t$ be the family of all
maximal intervals obtained in this way.

Two dyadic intervals are either disjoint or one contains the other.
Distinct members of $\mathcal J_t$ are therefore disjoint, and every
interval under consideration lies in exactly one member of
$\mathcal J_t$. In particular,
\[
 \sum_{J\in\mathcal J_t}|J|
 =\left|\bigcup_{J\in\mathcal J_t}J\right|\le |E_t|.
\]
Using \eqref{eq:descendants} in each maximal interval gives
\begin{align*}
 \sum_{\substack{I\in\D\\|I\setminus E_t|=0}}|I|^s
 &=\sum_{J\in\mathcal J_t}
       \sum_{\substack{I\in\D\\I\subset J}}|I|^s
 =\frac1{1-2^{1-s}}\sum_{J\in\mathcal J_t}|J|^s\\
 &\le\frac1{1-2^{1-s}}
       \left(\sum_{J\in\mathcal J_t}|J|\right)^s
 \le\frac{|E_t|^s}{1-2^{1-s}}.
\end{align*}

Write $a_I=\essinf_I g$. For each $I$,
\[
 a_I^s=s\int_0^\infty t^{s-1}\ind_{\{t<a_I\}}\,dt.
\]
If $a_I>t$, then $g>t$ almost everywhere on $I$. Tonelli's theorem
and the preceding bound therefore give
\begin{align*}
 \sum_I|I|^sa_I^s
 &=s\int_0^\infty t^{s-1}\sum_{a_I>t}|I|^s\,dt\\
 &\le\frac{s}{1-2^{1-s}}
       \int_0^\infty t^{s-1}|E_t|^s\,dt\\
       &\le\frac{s}{1-2^{1-s}}
       \int_0^\infty M^{s-1}|E_t|\,dt\\
       &=A_s M^s,
\end{align*}
 which proves \eqref{eq:one}.
\end{proof}

\begin{lemma}\label{lem:tensor}
Let $s>1$ and $g\ge0$ belong to $L^1(\R^n)$. Then
\begin{equation}\label{eq:tensor}
 \sum_{R\in\D^n}|R|^s(\essinf_R g)^s
 \le A_s^n\norm g_1^s.
\end{equation}
In particular, for every measurable $\Omega$ of finite measure,
\begin{equation}\label{eq:packing}
 \sum_{\substack{R\in\D^n\\R\subset\Omega}}|R|^s
 \le A_s^n|\Omega|^s.
\end{equation}
\end{lemma}

\begin{proof}
The case $n=1$ is Lemma \ref{lem:one}. Assume the result holds in
$n-1$ dimensions, and write $x=(x',t)\in\R^{n-1}\times\R$.
For each $J\in\D$, define, for almost every $x'$,
\[
 h_J(x')=\essinf_{t\in J}g(x',t),
 \qquad G(t)=\int_{\R^{n-1}}g(x',t)\,dx'.
\]
The function $h_J$ is measurable. Indeed, for every real $q$,
\[
 \{x':h_J(x')<q\}
 =\left\{x':\int_J\ind_{\{g(x',t)<q\}}\,dt>0\right\}
\]
up to a null set, and the integral on the right is a measurable
function of $x'$ by Tonelli's theorem. Also,
\[
 0\le h_J(x')\le\frac1{|J|}\int_Jg(x',t)\,dt,
\]
so $h_J\in L^1(\R^{n-1})$.

For every real $c$, Fubini's theorem gives
\[
 g\ge c\text{ a.e. on }R'\times J
 \quad\Longleftrightarrow\quad
 h_J\ge c\text{ a.e. on }R'.
\]
Taking the largest such lower bound proves the first identity below.
For the second, the defining inequality
$h_J(x')\le g(x',t)$ holds for almost every $(x',t)$ with $t\in J$.
Fubini's theorem allows us to integrate in $x'$ for almost every
$t\in J$, giving $\norm{h_J}_1\le G(t)$. Consequently,
\begin{equation}\label{eq:slices}
 \essinf_{R'\times J}g=\essinf_{R'}h_J,
 \qquad \norm{h_J}_1\le\essinf_JG.
\end{equation}
Tonelli's theorem also gives $G\in L^1(\R)$ and
$\norm G_1=\norm g_1$.

Every rectangle in $\D^n$ has a unique form $R'\times J$.
We first sum over $R'$, use the induction hypothesis for $h_J$,
and then apply Lemma \ref{lem:one} to $G$. Since all terms are
nonnegative, Tonelli's theorem permits each rearrangement:
\begin{align*}
 \sum_{R\in\D^n}|R|^s(\essinf_R g)^s
 &=\sum_{J\in\D}|J|^s
       \sum_{R'\in\D^{n-1}}|R'|^s(\essinf_{R'}h_J)^s\\
 &\le A_s^{n-1}\sum_{J\in\D}|J|^s\norm{h_J}_1^s\\
 &\le A_s^{n-1}\sum_{J\in\D}|J|^s(\essinf_JG)^s\\
 &\le A_s^n\norm G_1^s
 =A_s^n\norm g_1^s.
\end{align*}
For $g=\ind_\Omega$, we derive  \eqref{eq:packing} directly from
\eqref{eq:tensor}.
\end{proof}

\section{Proof of the embedding}

For $f\in\Bp$ and $R\in\D^n$, write
\[
 c_R=\ip f{w_R},
 \qquad d_R=|R|^{-1/2}c_R.
\]
The pairing is computed from the finitely many frequency blocks
meeting the Fourier support of $w_R$, as in the proof below.

\begin{lemma}\label{lem:coeff}
If $f\in\Bp$ and $1\le p<\infty$, then
\begin{equation}\label{eq:coeff}
 \sum_{R\in\D^n}|d_R|^p\lesssim_{n,p}\norm f_{\Bp}^p.
\end{equation}
\end{lemma}

\begin{proof}
The Fourier support of $w_{\kk,\mm}$ is contained in
\[
 a2^{k_\ell}\le|\xi_\ell|\le b2^{k_\ell},
 \qquad 1\le\ell\le n.
\]
Since $f$ has positive frequencies, only the part with
$\xi_\ell>0$ can contribute. On the support of $\Psi_\jj$, we have
$2^{j_\ell-1}\le\xi_\ell\le2^{j_\ell+1}$. If these two supports
meet, then
\[
 2^{j_\ell-1}\le b2^{k_\ell},
 \qquad a2^{k_\ell}\le2^{j_\ell+1},
\]
and hence
$\log_2a-1\le j_\ell-k_\ell\le\log_2b+1$.
We may therefore choose a fixed integer $L\ge0$ such that every
block that can contribute satisfies $|j_\ell-k_\ell|\le L$ in
each coordinate. 

Let
\[
 \mathcal N(\kk)=\{\jj:|j_\ell-k_\ell|\le L\text{ for all }\ell\},
 \qquad M=(2L+1)^n,
 \qquad u_\kk=\sum_{\jj\in\mathcal N(\kk)}\Delta_\jj f.
\]

It follows that
$c_{\kk,\mm}=\ip{u_\kk}{w_{\kk,\mm}}$. With the normalization of
$d_{\kk,\mm}$, it becomes
\[
 d_{\kk,\mm}
 =2^{\sigma(\kk)}
   \int_{\R^n}u_\kk(x)\overline{W(2^\kk x-\mm)}\,dx.
\]

Set $A=\norm W_1$. The measure
\[
 d\mu_{\kk,\mm}(x)
 =2^{\sigma(\kk)}|W(2^\kk x-\mm)|\,dx
\]
has total mass $A$.
For $p>1$, the triangle inequality followed by H\"older's inequality
on this measure gives
\[
 |d_{\kk,\mm}|
 \le\int|u_\kk|\,d\mu_{\kk,\mm}
 \le\left(\int|u_\kk|^p\,d\mu_{\kk,\mm}\right)^{1/p}
      A^{1-1/p}.
\]
Thus
\begin{equation}\label{eq:weighted}
 |d_{\kk,\mm}|^p
 \le A^{p-1}2^{\sigma(\kk)}
      \int |u_\kk(x)|^p|W(2^\kk x-\mm)|\,dx.
\end{equation}
For $p=1$, the same inequality is apparent.

We next show that
\[
 P_W:=\sup_{y\in\R^n}
          \sum_{\mm\in\Z^n}|W(y-\mm)|<\infty.
\]
The sum is unchanged when $y$ is translated by an integer vector,
so it is enough to take $y\in[0,1]^n$.
Choose $B>n$. Schwartz decay and
$1+|\mm|\le(1+\sqrt n)(1+|y-\mm|)$ give
\[
 |W(y-\mm)|\le C_B(1+|y-\mm|)^{-B}
 \le C_{B,n}(1+|\mm|)^{-B}
\]
uniformly on this cube. There are at most $C_n2^{rn}$ lattice
points with $|\mm|<2^{r+1}$, so
\[
 \sum_{\mm\in\Z^n}(1+|\mm|)^{-B}
 \le C_n\left(1+\sum_{r=0}^\infty2^{rn}2^{-rB}\right)<\infty.
\]
Summing \eqref{eq:weighted} over $\mm$ and applying Tonelli's
theorem therefore yields
\[
 \sum_\mm|d_{\kk,\mm}|^p
 \le A^{p-1}P_W2^{\sigma(\kk)}\norm{u_\kk}_p^p.
\]

H\"older's inequality for a sum of $M$ numbers gives
\[
 \left|\sum_{\jj\in\mathcal N(\kk)}\Delta_\jj f(x)\right|^p
 \le M^{p-1}\sum_{\jj\in\mathcal N(\kk)}
                  |\Delta_\jj f(x)|^p.
\]
After integration and summation over $\kk$, we obtain
\[
 \sum_{\kk,\mm}|d_{\kk,\mm}|^p
 \le A^{p-1}P_WM^{p-1}
       \sum_\kk\sum_{\jj\in\mathcal N(\kk)}
       2^{\sigma(\kk)}\norm{\Delta_\jj f}_p^p.
\]
For a fixed $\jj$, the condition $\jj\in\mathcal N(\kk)$ allows
$2L+1$ choices of $k_\ell$ in each coordinate, hence exactly $M$
choices of $\kk$. For each one,
$\sigma(\kk)-\sigma(\jj)\le nL$. Reversing the order of the
nonnegative sums now gives
\begin{align*}
 \sum_\kk\sum_{\jj\in\mathcal N(\kk)}
       2^{\sigma(\kk)}\norm{\Delta_\jj f}_p^p
 &=\sum_\jj\norm{\Delta_\jj f}_p^p
       \sum_{\kk:\,\jj\in\mathcal N(\kk)}2^{\sigma(\kk)}\\
 &\le M2^{nL}\sum_\jj2^{\sigma(\jj)}
                        \norm{\Delta_\jj f}_p^p.
\end{align*}
Combining the last two estimates proves
\[
 \sum_{\kk,\mm}|d_{\kk,\mm}|^p
 \le A^{p-1}P_WM^p2^{nL}\norm f_{\Bp}^p,
\]
which gives \eqref{eq:coeff}.
\end{proof}

\begin{proof}[Proof of Theorem \ref{thm:embedding}]
Let $f\in\Bp$, and fix an open set $\Omega$ with
$0<|\Omega|<\infty$. Suppose $p>2$ and set $s=p/(p-2)>1$.
H\"older's inequality with exponents $p/2$ and $s$, followed by
Lemmas \ref{lem:tensor} and \ref{lem:coeff}, gives
\begin{align}\label{eq:carlesonbound}
 \sum_{R\subset\Omega}|c_R|^2
 &=\sum_{R\subset\Omega}|d_R|^2|R|\notag\\
 &\le
 \left(\sum_{R\subset\Omega}|d_R|^p\right)^{2/p}
 \left(\sum_{R\subset\Omega}|R|^s\right)^{1/s}\notag\\
 &\le A_s^{n/s}|\Omega|\,
       \norm{(d_R)}_{\ell^p}^2
 \lesssim_{n,p}|\Omega|\norm f_{\Bp}^2.
\end{align}
For $1\le p\le2$, use $|R|\le|\Omega|$ and
$\ell^p\subset\ell^2$ to obtain
\[
 \frac1{|\Omega|}\sum_{R\subset\Omega}|c_R|^2
 \le\sum_R|d_R|^2
 \le\norm{(d_R)}_{\ell^p}^2
 \lesssim_{n,p}\norm f_{\Bp}^2.
\]
By the wavelet characterization \eqref{eq:carleson}, $f$ has a
representative in $P_+\BMO$, with
\[
 \norm f_{\BMO}\lesssim_{n,p}\norm f_{\Bp}.
\]
\end{proof}

\begin{remark}\label{rem:endpoint}
For the unit dyadic cube $Q=[0,1)^n$,
\[
 \sum_{\substack{R\in\D^n\\R\subset Q}}|R|^s
 =\prod_{\ell=1}^n\sum_{k_\ell=0}^\infty
       2^{-(s-1)k_\ell}
 =(1-2^{1-s})^{-n}.
\]
Thus the packing constant diverges as $s\downarrow1$, and
\eqref{eq:carlesonbound} gives no estimate at $p=\infty$.
\end{remark}

\section{Fej\'er approximation on the polydisc}\label{sec:fejer}

Let $\mathbb D=\{z\in\C:|z|<1\}$ and
$\mathbb T=\partial\mathbb D$. We work on $H^2(\mathbb D^n)$,
identified with its boundary functions on $\mathbb T^n$.
In angular coordinates $t=(t_1,\ldots,t_n)$, normalized Haar measure
is $dm(t)=(2\pi)^{-n}\,dt_1\cdots dt_n$. We use the product Hardy space $\Hprod(\mathbb T^n)$
and its dual $\BMO(\mathbb T^n)$, including constants and functions
of fewer than $n$ variables. The BMO norm in this section is the
dual norm.

Identify $H^2(\mathbb D^n)$ with $\ell^2(\mathbb N_0^n)$ through the
orthonormal basis $e_\alpha(z)=z^\alpha$, where
$\mathbb N_0=\{0,1,2,\ldots\}$,
$\alpha=(\alpha_1,\ldots,\alpha_n)$, and
$z^\alpha=\prod_{\ell=1}^nz_\ell^{\alpha_\ell}$.
The vector $e_0$ is the constant function one.
We use the linear Hankel matrix convention
\begin{equation}\label{eq:diskmatrix}
 (H_\varphi)_{\alpha,\beta}=a_{\alpha+\beta},
 \qquad
 \varphi(z)=\sum_{\gamma\in\mathbb N_0^n}a_\gamma z^\gamma.
\end{equation}
If $H_\varphi$ is bounded, then
$(a_\alpha)_\alpha=H_\varphi e_0\in\ell^2$, so its analytic
symbol belongs to $H^2(\mathbb D^n)$.

Write
\[
 \mathrm{BMOA}(\mathbb D^n)
 =H^2(\mathbb D^n)\cap\BMO(\mathbb T^n),
\]
and let $\mathrm{VMOA}(\mathbb D^n)$ be the closure of the analytic
polynomials in this BMO norm.

For $N\ge0$, set
\[
 F_N(t)=\sum_{|k|\le N}
       \left(1-\frac{|k|}{N+1}\right)e^{ikt},
 \qquad K_N(t)=\prod_{\ell=1}^nF_N(t_\ell).
\]
The kernels $K_N$ form a nonnegative approximate identity with integral one.
Write
\[
 \varphi_N=\varphi*K_N,\qquad
 w_N(\gamma)=\prod_{\ell=1}^n
          \left(1-\frac{\gamma_\ell}{N+1}\right)_+.
\]
Here $r_+=\max\{r,0\}$ and convolution is taken on $\mathbb T^n$:
$\varphi_N(e^{it})=\int_{\mathbb T^n}\varphi(e^{i(t-s)})K_N(s)\,dm(s)$.
Thus $\widehat{\varphi_N}(\gamma)=w_N(\gamma)a_\gamma$ for
$\gamma\in\mathbb N_0^n$, and
\[
 (H_{\varphi_N})_{\alpha,\beta}
       =w_N(\alpha+\beta)a_{\alpha+\beta}.
\]
This matrix vanishes if any coordinate of its row or column index
exceeds $N$. Hence
\[
 \operatorname{rank}H_{\varphi_N}\le(N+1)^n,
 \qquad H_{\varphi_N}\in S^p\quad(1\le p<\infty).
\]
\begin{proposition}\label{prop:fejer}
Let $H_\varphi$ be bounded on $H^2(\mathbb D^n)$. Then
\begin{equation}\label{eq:fejerbound}
 \norm{H_{\varphi_N}}_{\mathrm{op}}
          \le\norm{H_\varphi}_{\mathrm{op}}.
\end{equation}
The operators $H_{\varphi_N}$ converge to $H_\varphi$ in the weak
operator topology.
\end{proposition}

\begin{proof}
Let $U_te_\alpha=e^{i\alpha\cdot t}e_\alpha$. Direct computation of
matrix entries, using that $K_N$ is even, gives
\begin{equation}\label{eq:rotationaverage}
 H_{\varphi_N}
 =\int_{\mathbb T^n}K_N(t)\,U_tH_\varphi U_t\,dm(t).
\end{equation}
The integral is taken on each vector. Since $U_t$ is unitary and
$K_N$ is nonnegative with integral one,
\[
 \norm{H_{\varphi_N}f}_2
 \le\int_{\mathbb T^n}K_N(t)\norm{U_tH_\varphi U_tf}_2\,dm(t)
 \le\norm{H_\varphi}_{\mathrm{op}}\norm f_2.
\]
This proves \eqref{eq:fejerbound}.

For every pair of monomials, the matrix formula gives
\[
 \ip{H_{\varphi_N}e_\beta}{e_\alpha}
 \longrightarrow a_{\alpha+\beta}
 =\ip{H_\varphi e_\beta}{e_\alpha}.
\]
By linearity, this holds for analytic polynomials $f,g$. Their density
in $H^2$ and the uniform bound \eqref{eq:fejerbound} then give
\[
 \ip{H_{\varphi_N}f}g\longrightarrow\ip{H_\varphi f}g,
 \qquad f,g\in H^2.
\]
This proves weak operator convergence.
\end{proof}

\begin{question}\label{question:polydisc}
Let $n\ge2$ and $\varphi\in H^2(\mathbb D^n)$. Are the following
equivalences valid for small Hankel operators on $H^2(\mathbb D^n)$?
\[
 \begin{aligned}
 H_\varphi\text{ is bounded}
 &\quad\Longleftrightarrow\quad
       \varphi\in\mathrm{BMOA}(\mathbb D^n),\\
 H_\varphi\text{ is compact}
 &\quad\Longleftrightarrow\quad
       \varphi\in\mathrm{VMOA}(\mathbb D^n).
 \end{aligned}
\]
\end{question}

The BMOA condition is sufficient for boundedness by Hardy--BMO
duality. If the symbol is in VMOA, approximation by analytic
polynomials gives compactness. The question is whether these
conditions are necessary.

Each $H_{\varphi_N}$ has finite rank, so it belongs to every
$S^p$, $1\le p<\infty$. We can apply Corollary \ref{cor:hankel}
in the polydisc model as well. Indeed, the coordinatewise Cayley
map $\kappa(x)=(x-i)/(x+i)$, with the usual Hardy-space unitaries,
carries Hankel forms between the two models and preserves their
Schatten norms. Bounded symbols pass between the models by
composition and multiplication by a factor of modulus one.
By analytic Hardy--BMO duality and the Hahn--Banach theorem, a
product BMO symbol has a bounded symbol with the same analytic
projection. These facts transfer the estimate in
Corollary \ref{cor:hankel} to the polydisc and give
\[
 \norm{\varphi_N}_{\BMO(\mathbb T^n)}
 \le C_{n,p}\norm{H_{\varphi_N}}_{S^p}<\infty
 \qquad\text{for every }N.
\]
Also, each $\varphi_N$ is an analytic polynomial, so it belongs to
$\mathrm{VMOA}(\mathbb D^n)$.

The estimate from Corollary \ref{cor:hankel} involves the Schatten
norm, whereas Proposition \ref{prop:fejer} controls only the
operator norm. Thus membership of each $\varphi_N$ in BMOA does
not provide a bound for
$\sup_N\norm{\varphi_N}_{\BMO(\mathbb T^n)}$.
Such a bound is what the following limit argument needs.

Moreover, if
\begin{equation}\label{eq:uniformbmo}
 M:=\sup_N\norm{\varphi_N}_{\BMO(\mathbb T^n)}<\infty,
\end{equation}
then $\varphi\in\BMO(\mathbb T^n)$ and
$\norm\varphi_{\BMO(\mathbb T^n)}\le M$.
To prove this, first note that $0\le w_N(\gamma)\le1$ and
$w_N(\gamma)\to1$. Since $\varphi\in H^2$, dominated convergence gives
\[
 \norm{\varphi_N-\varphi}_2^2
 =\sum_{\gamma\in\mathbb N_0^n}
      |w_N(\gamma)-1|^2|a_\gamma|^2\longrightarrow0.
\]
For every trigonometric polynomial $g$, it follows that
\[
 \left|\int_{\mathbb T^n}\varphi\overline g\,dm\right|
 =\lim_N\left|\int_{\mathbb T^n}\varphi_N\overline g\,dm\right|
 \le M\norm g_{\Hprod(\mathbb T^n)}.
\]
By density of trigonometric polynomials in $\Hprod(\mathbb T^n)$,
this pairing extends to a bounded functional of norm at most $M$.
Hardy--BMO duality gives the assertion.

Now suppose that $H_\varphi$ is compact. The dual of the space of
compact operators is the trace class, under the trace pairing.
Let $T$ be trace class, and choose finite-rank $T_k$ with
$\norm{T-T_k}_{S^1}\to0$. We have
\begin{align*}
 &\left|\operatorname{Tr}\bigl(T^*(H_{\varphi_N}-H_\varphi)\bigr)\right|\\
 &\quad\le
 2\norm{H_\varphi}_{\mathrm{op}}\norm{T-T_k}_{S^1}
 +\left|\operatorname{Tr}\bigl(T_k^*(H_{\varphi_N}-H_\varphi)\bigr)\right|.
\end{align*}
For fixed $k$, the last term tends to zero by weak operator
convergence. Letting first $N\to\infty$ and then $k\to\infty$ shows
that $H_{\varphi_N}$ converges weakly to $H_\varphi$ in the space
of compact operators. Mazur's lemma gives finite convex combinations
that converge in operator norm. Thus one can choose
\[
 q_j=\sum_{N=j}^{m_j}\lambda_{j,N}\varphi_N,
 \qquad \lambda_{j,N}\ge0,\qquad
 \sum_{N=j}^{m_j}\lambda_{j,N}=1,
\]
with $\norm{H_{q_j}-H_\varphi}_{\mathrm{op}}\to0$.

Every $q_j$ is a polynomial and hence belongs to VMOA. We also have
\[
 \norm{q_j-\varphi}_2
 =\norm{(H_{q_j}-H_\varphi)e_0}_2
 \le\norm{H_{q_j}-H_\varphi}_{\mathrm{op}}\longrightarrow0.
\]
This does not yet give convergence in BMO norm. Even a uniform BMO
bound gives only the BMO conclusion proved above. To obtain VMOA
from these polynomial approximants, it would suffice to prove that
$(q_j)$ is Cauchy in BMO norm. Its BMO limit would then agree with
its $H^2$ limit $\varphi$, by testing against trigonometric
polynomials. The definition of VMOA would give the result.

Both limit arguments would follow from a bound
\begin{equation}\label{eq:polynomialbound}
 \norm P_{\BMO(\mathbb T^n)}\le C_n\norm{H_P}_{\mathrm{op}}
\end{equation}
for analytic polynomials, with $C_n$ independent of the degree.
For bounded $H_\varphi$, apply this bound to $\varphi_N$ and use
\eqref{eq:fejerbound} to obtain \eqref{eq:uniformbmo}.
For compact $H_\varphi$, apply it to the polynomial $q_j-q_k$:
\[
 \norm{q_j-q_k}_{\BMO(\mathbb T^n)}
 \le C_n\norm{H_{q_j}-H_{q_k}}_{\mathrm{op}}
 \longrightarrow0\qquad(j,k\to\infty).
\]
This would give $\varphi\in\mathrm{VMOA}(\mathbb D^n)$.
Thus the uniform symbol estimate \eqref{eq:polynomialbound} is the
step still needed in this approach.


\begin{thebibliography}{99}

\bibitem{Nehari}
Z. Nehari,
\emph{\href{https://doi.org/10.2307/1969670}{On bounded bilinear forms}},
Ann. of Math. (2) \textbf{65} (1957), 153--162.

\bibitem{Hartman}
P. Hartman,
\emph{On completely continuous Hankel matrices},
Proc. Amer. Math. Soc. \textbf{9} (1958), 862--866.
\href{https://doi.org/10.1090/S0002-9939-1958-0108684-8}
{doi:10.1090/S0002-9939-1958-0108684-8}.

\bibitem{Peller1}
V. V. Peller,
\emph{Hankel operators of class $S_p$ and their applications
(rational approximation, Gaussian processes, the problem of
majorizing operators)},
Math. USSR-Sb. \textbf{41} (1982), 443--479.
\href{https://www.mathnet.ru/eng/sm2817}{Publisher's record}.

\bibitem{Peller2}
V. V. Peller,
\emph{A description of Hankel operators of class $S_p$ for $p>0$,
an investigation of the rate of rational approximation, and other
applications},
Math. USSR-Sb. \textbf{50} (1985), 465--494.
\href{https://www.mathnet.ru/eng/sm2310}{Publisher's record}.

\bibitem{FL}
S. H. Ferguson and M. T. Lacey,
\emph{A characterization of product BMO by commutators},
Acta Math. \textbf{189} (2002), 143--160.
\href{https://arxiv.org/abs/math/0104144}{arXiv:math/0104144}.

\bibitem{LT}
M. T. Lacey and E. Terwilleger,
\emph{Hankel operators in several complex variables and product BMO},
Houston J. Math. \textbf{35} (2009), 159--183.
\href{https://arxiv.org/abs/math/0310348}{arXiv:math/0310348}.

\bibitem{LTW}
M. T. Lacey, E. Terwilleger, and B. D. Wick,
\emph{Remarks on product VMO},
Proc. Amer. Math. Soc. \textbf{134} (2006), 465--474.
\href{https://arxiv.org/abs/math/0405097}{arXiv:math/0405097}.

\bibitem{HTV}
I. Holmes, S. Treil, and A. Volberg,
\emph{Dyadic bi-parameter repeated commutator and dyadic product BMO},
\href{https://arxiv.org/abs/2101.00763}{arXiv:2101.00763}, 2021.

\bibitem{PS}
S. Pott and M. P. Smith,
\emph{Paraproducts and Hankel operators of Schatten class via
$p$-John--Nirenberg theorem},
J. Funct. Anal. \textbf{217} (2004), 38--78.
\href{https://doi.org/10.1016/j.jfa.2004.03.005}
{doi:10.1016/j.jfa.2004.03.005}.

\bibitem{BPBesov}
K. Bampouras and K.-M. Perfekt,
\emph{Besov spaces and Schatten class Hankel operators for Hardy and
Paley--Wiener spaces in higher dimensions},
\href{https://arxiv.org/abs/2409.04184v6}{arXiv:2409.04184v6}.

\bibitem{BPRecover}
K. Bampouras and K.-M. Perfekt,
\emph{Recovering product BMO from Schatten Hankel operators},
\href{https://arxiv.org/abs/2604.26722v2}{arXiv:2604.26722v2}.

\bibitem{DP}
L. Dalenc and S. Petermichl,
\emph{A lower bound criterion for iterated commutators},
J. Funct. Anal. \textbf{266} (2014), 5300--5320.
\href{https://arxiv.org/pdf/1307.6423}{arXiv:1307.6423}.

\end{thebibliography}
\end{document}